\documentclass[12pt]{article}
\title{Tropical Duality in $(d+2)$-angulated categories}
\author{Joseph Reid}
\date{}
\usepackage{fancyhdr}
\usepackage{amsfonts}
\usepackage{amsthm}
\usepackage{mathtools}
\usepackage{amsmath}
\usepackage{amssymb}
\usepackage{wasysym}
\usepackage{mathrsfs}
\usepackage[margin=1.1in]{geometry}
\theoremstyle{definition}

\usepackage[T1]{fontenc}

\newtheorem{theorem}{Theorem}[section]
\newtheorem*{theorem*}{Theorem}

\newtheorem{proposition}[theorem]{Proposition}
\newtheorem{lemma}[theorem]{Lemma}

\newtheorem{definition}[theorem]{Definition}

\usepackage{tikz}
\usetikzlibrary{arrows,positioning}
\usetikzlibrary{decorations.markings}

\tikzset{
  big dot/.style={
    circle, inner sep=0pt, 
    minimum size=3mm, fill=black
 }
}
\tikzset{
  normal dot/.style={
    circle, inner sep=0pt, 
    minimum size=1.5mm, fill=black
 }
}

\newcommand{\splitG}{K^{\textrm{split}}_0}
\newcommand{\indT}{\textrm{Ind}_\mathscr{T}}
\newcommand{\indU}{\textrm{Ind}_\mathscr{U}}
\newcommand{\cvecU}{c_\mathscr{T}(u,\mathscr{U})}
\newcommand{\indec}{\textrm{Indec}}

\newcommand{\homs}{\textrm{Hom}_{\mathscr{C}}}
\newcommand{\tilthoms}{\textrm{Hom}_{\frac{\mathscr{C}}{[\Sigma^d \mathscr{T}]}}}
\newcommand{\tilthomsU}{\textrm{Hom}_{\frac{\mathscr{C}}{[\Sigma^d \mathscr{U}]}}}
\newcommand{\add}{\textrm{add}}

\begin{document}
\maketitle
ABSTRACT. Let $\mathscr{C}$ be a $2$-Calabi-Yau triangulated category with two cluster tilting subcategories $\mathscr{T}$ and $\mathscr{U}$. A result from \cite{JorgYak, DIJ} known as tropical duality says that the index with respect to $\mathscr{T}$ provides an isomorphism between the split Grothendieck groups of $\mathscr{U}$ and $\mathscr{T}$. We also have the notion of $c$-vectors, which using tropical duality have been proven to have sign coherence, and to be recoverable as dimension vectors of modules in a module category. \\
The notion of triangulated categories extends to the notion of $(d+2)$-angulated categories. Using a higher analogue of cluster tilting objects, this paper generalises tropical duality to higher dimensions. This implies that these basic cluster tilting objects have the same number of indecomposable summands. It also proves that under conditions of mutability, $c$-vectors in the $(d+2)$-angulated case have sign coherence, and shows formulae for their computation. Finally, it proves that under the condition of mutability, the $c$-vectors are recoverable as dimension vectors of modules in a module category.
\section{Introduction}
\thispagestyle{fancy}
Let $\mathscr{C}$ be a triangulated category with certain nice properties. The notion of a cluster tilting subcategory of $\mathscr{C}$ is due to \cite[Definition~2.2]{Iyama}, and we can define the index with respect to a cluster tilting subcategory \cite[Section~2.1]{Palu}. The index has several useful properties that aid computation and comparison of cluster tilting subcategories. Thanks to \cite{JorgYak, DIJ}, we have an isomorphism which we name Tropical Duality:
\begin{theorem*}[Yakimov-J{\o}rgensen]\cite[Theorem~1.2]{JorgYak}\cite[Cor.~6.20]{DIJ}\label{origDual}
Suppose that $\mathscr{C}$ is $2$-Calabi-Yau, $K$-linear, Hom-finite and Krull-Schmidt. For every pair of cluster tilting subcategories $\mathscr{T}$ and $\mathscr{U}$, there are inverse isomorphisms
\begin{center}
\begin{tikzpicture}[line cap = round, line join = round]
\node (a) at (-3, 0) {$K_0^{\textrm{split}}(\mathscr{T})$};
\node (b) at (3, 0) {$K_0^{\textrm{split}}(\mathscr{U}).$};

%\draw ([yshift=0.8 cm]a.east) -- (b);
\draw [->] ([yshift= 0.2 cm]a.east) to node[above]{$-\textrm{ind}_{\mathscr{U}}\circ \Sigma |_{K_0^{\textrm{split}}(\mathscr{T})}$} ([yshift= 0.2 cm]b.west);
\draw [->] ([yshift= -0.2 cm]b.west) to node[below]{$\textrm{ind}_{\mathscr{T}}|_{K_0^{\textrm{split}}(\mathscr{U})}$} ([yshift= -0.2 cm]a.east);
\end{tikzpicture}
\end{center}
\end{theorem*}

This implies, as shown already by Dehy and Keller \cite{DehyKeller}, that all cluster tilting subcategories of $\mathscr{C}$ have the same number of indecomposable objects. \\

We have the notion of homological $c$-vectors with respect to these objects, as defined in \cite[Definition~2.8]{JorgYak}. J{\o}rgensen and Yakimov proved in \cite[Theorem~1.2(2)]{JorgYak} that these $c$-vectors can be obtained as dimension vectors in the module category of the endomorphism ring of a cluster tilting object, which generalises work done by A. N\'{a}jera Ch\'{a}vez \cite{Chavez}. \\

In this paper we will generalise these results into the higher homological case. We recall some important definitions before stating these results. Instrumental to everything we do here are the notions of Oppermann-Thomas cluster tilting subcategory and index. These definitions require a $(d+2)$-angulated category as defined by Geiss, Keller, and Oppermann \cite{GKO}. We will recap this in section 2. For the following definitions, we let ($\mathscr{C}, \Sigma^d, \pentagon$) be a $(d+2)$-angulated category.
\begin{definition}\cite[Definition~5.3]{OppermannThomas}\label{OTDef}
Let $\mathscr{C}$ be a $(d+2)$-angulated category, and let $T \in \mathscr{T}$, where $\mathscr{T} = \textrm{add}(T)$ is the corresponding additive subcategory of $\mathscr{C}$. We call $T$ an \textit{Oppermann-Thomas cluster tilting object} of $\mathscr{C}$ if:
\begin{itemize}
\item[(i)] $\homs{}(\mathscr{T}, \Sigma^d(\mathscr{T})) = 0$,
\item[(ii)] for any $c \in \mathscr{C}$, there exists a $(d+2)$-angle
\begin{align}
t_d \rightarrow t_{d-1} \rightarrow \cdots \rightarrow t_1 \rightarrow t_0 \rightarrow c \rightarrow \Sigma^d(t_d) \label{theseq}
\end{align}
where $t_i \in \mathscr{T}$ for each $i$.
\end{itemize}
In this case, $\mathscr{T} = \textrm{add}(T)$ is an \textit{Oppermann-Thomas cluster tilting subcategory}.
\end{definition}

If we have an Oppermann-Thomas cluster tilting subcategory $\mathscr{T} = \textrm{add}(T)$ we may construct the \textit{split Grothendieck group} for $\mathscr{T}$, which we denote $\splitG{}(\mathscr{T})$. This group is the abelian group generated by the objects of $\mathscr{T}$, modulo all the relations of the form $[t]=[t_0] + [t_1]$ where $t \cong t_0 \oplus t_1$. This gives us the following formula:
\begin{align*}
\splitG{}(\mathscr{T}) = \coprod_{t \in \indec{}(\mathscr{T})} \mathbb{Z} \cdot [t]
\end{align*}

Using this, we may define the notion of index:

\begin{definition}\cite[Definition~B]{JustJorg}
The \textit{index} of an object $c \in \mathscr{C}$ with respect to an Oppermann-Thomas cluster tilting subcategory $\mathscr{T}$ is defined as:
\begin{align*}
\textrm{Ind}_{\mathscr{T}}(c) = \Sigma_{i=0}^d (-1)^i[t_i]
\end{align*}
where
\begin{align*}
t_d \rightarrow t_{d-1} \rightarrow \cdots \rightarrow t_1 \rightarrow t_0 \rightarrow c \rightarrow \Sigma^d(t_d)
\end{align*}
is a $(d+2)$-angle with each $t_i \in \mathscr{T}$. It follows from \cite[Remark~5.4]{JustJorg} that the index is well defined when $\mathscr{C}$ is Hom-finite with split idempotents.
\end{definition}

We introduce some notation that we will use throughout. Let $\mathscr{C}$ be a $(d+2)$-angulated category, and let $T$ be an Oppermann-Thomas cluster tilting object with $\Gamma_T = \textrm{End}(T)$. Then we can define a functor $F_T: \mathscr{C} \to \textrm{mod }\Gamma_T$ that acts by sending $x \in \mathscr{C}$ to $\homs{}(T, x)$.

We pause here to note that unlike in the classic case, there are cluster tilting subcategories in the higher case which are not mutable. We define mutability in the following way: \\
Let $U$ be a basic Oppermann-Thomas cluster tilting object of $\mathscr{C}$, and let $\{u_1, u_2, \ldots, u_m\}$ be the set of indecomposable summands of $U$. We say that $U$ is \textit{mutable} at the indecomposable summand $u \in \{u_1, u_2, \ldots, u_m\}$ if there is an indecomposable object $u^* \in \mathscr{C}$ such that the object with indecomposable summands $(\{u_1, u_2, \ldots, u_m\} \textbackslash u) \cup u^*$ is also an Oppermann-Thomas cluster tilting object. In this case we call $u^*$ a \textit{mutation} of $u$. We can then make the following definition.
\begin{definition}\label{exchangePair}
Let $\mathscr{C}$ be a $(d+2)$-angulated category, let $U$ be a basic Oppermann-Thomas cluster tilting object with $\mathscr{U} = \textrm{add}(U)$. Suppose that $U$ is mutable at $u$ with the mutation $u^*$. We call $u$ and $u^*$ an \textit{exchange pair} if $\textrm{Ext}^d(u, u^*)$ and $\textrm{Ext}^d(u^*, u)$ both have dimension 1 over $K$, and there exist two $(d+2)$-angles
\begin{align}\label{exchangeOne}
u^* \to e_d \to \ldots \to e_1 \to u \xrightarrow{\delta} \Sigma^d u^*
\end{align}
and
\begin{align}\label{exchangeTwo}
u \to f_d \to \ldots \to f_1 \to u^* \xrightarrow{\epsilon} \Sigma^d u,
\end{align}
where each $e_i$ and $f_i$ is a sum of indecomposable summands from $(\indec{}(\mathscr{U}) \backslash u)$.
\end{definition}
We note at this point that definition \ref{exchangePair} contains strong assumptions. In the Oppermann-Thomas $(d+2)$-angulated categories of Dynkin type $A_n$, these assumptions are all met and these exchange pairs exist. See section 5 for more details.

We fix some more terminology:
\begin{definition}\label{boilerplateDef}
We will often consider the following setup: $\mathscr{C}$ is a $2d$-Calabi-Yau $(d+2)$-angulated category that is $K$-linear, Hom-finite, and Krull-Schmidt. Let $T$ and $U$ be two basic Oppermann-Thomas cluster tilting objects of $\mathscr{C}$, with corresponding subcategories $\mathscr{T} = \add{}(T)$ and $\mathscr{U} = \add{}(U)$. We let $\Gamma_T = \textrm{End}(T)$ and $\Gamma_U = \textrm{End}(U)$, and let the functors $F_T$ and $F_U$ be defined as above.
\end{definition}

Finally, we can define homological $c$-vectors and $g$-vectors. For an abelian group $A$, we set
\begin{align*}
A^*:= \textrm{Hom}(A, \mathbb{Z}).
\end{align*}
If $T$ is basic and $\add{}(T)=\mathscr{T} \subseteq \mathscr{C}$ is an Oppermann-Thomas cluster tilting subcategory, then
\begin{align*}
\splitG{}(\mathscr{T})^* = \prod_{t \in \indec{}(\mathscr{T})} \mathbb{Z} \cdot [t]^*
\end{align*}
where $[t]^* \in  \splitG{}(\mathscr{T})^*$ is the unique element defined by
\begin{align*}
[t]^*([s]) = \delta_{ts} \textrm{ for all } s \in \indec{}(\mathscr{T}).
\end{align*}

\begin{definition}
Let $\mathscr{C}$, $T$, and $U$ be as in Definition \ref{boilerplateDef}. For $u \in \indec{}(\mathscr{U})$, we define the \textit{homological c-vector} of $(u, \mathscr{U})$ with respect to $\mathscr{T}$ to be the element $\cvecU{} \in \splitG{}(\mathscr{T})^*$ such that
\begin{align*}
\cvecU{}(\textrm{Ind}_\mathscr{T}(v)) = \delta_{uv} \textrm{ for each } v \in \textrm{Indec}(\mathscr{U}).
\end{align*}
\end{definition}

By Theorem $\ref{dknew}$, we see that $\cvecU{}$ exists and is unique. \\

\begin{definition}
Let $\mathscr{C}$, $T$, and $U$ be as in Definition \ref{boilerplateDef}. For $u \in \indec{}(\mathscr{U})$, we define the \textit{homological g-vector} of $u$ with respect to $\mathscr{T}$ to be the element $g_\mathscr{T}(u) = \indT{}(u)$ of $K_0^{\textrm{split}}(\mathscr{T})$.
\end{definition}

If the set $\{ e_1, e_2, \ldots, e_n \}$ is a basis of the free group $A$, then the dual basis $\{\epsilon_1, \epsilon_2 , \ldots, \epsilon_n \}$ of $a^* := \textrm{Hom}_\mathbb{Z}(A, \mathbb{Z})$ is defined by $\epsilon_i(e_j) = \delta_{ij}$.
By Theorem \ref{dknew} the $g$-vectors $\{g_\mathscr{T}(u) = \textrm{Ind}_\mathscr{T}(u)| u \in \indec{}(\mathscr{U}) \}$ are a basis of $K_0^{\textrm{split}}(\mathscr{T})$. The $c$-vectors $\{\cvecU{}| u \in \textrm{indec}(\mathscr{U}) \}$ are the dual basis of $K_0^{\textrm{split}}(\mathscr{T})^*$.\\

Having made these definitions, we state here the three main results of this paper.

\begingroup
\renewcommand{\thetheorem}{\Alph{theorem}}
\setcounter{theorem}{0}
\begin{theorem}[= Theorem \ref{tropDual}]\label{ThmA}
Let $\mathscr{C}$, $T$, and $U$ be as in Definition \ref{boilerplateDef}. Then there are inverse isomorphisms
\begin{center}
\begin{tikzpicture}[line cap = round, line join = round]
\node (a) at (-3, 0) {$K_0^{\textrm{split}}(\mathscr{T})$};
\node (b) at (3, 0) {$K_0^{\textrm{split}}(\mathscr{U}).$};

%\draw ([yshift=0.8 cm]a.east) -- (b);
\draw [->] ([yshift= 0.2 cm]a.east) to node[above]{$(-1)^d\textrm{ind}_{\mathscr{U}} \circ \Sigma^d|_{K_0^{\textrm{split}}(\mathscr{T})}$} ([yshift= 0.2 cm]b.west);
\draw [->] ([yshift= -0.2 cm]b.west) to node[below]{$\textrm{ind}_{\mathscr{T}}|_{K_0^{\textrm{split}}(\mathscr{U})}$} ([yshift= -0.2 cm]a.east);
\end{tikzpicture}
\end{center}
\end{theorem}

\begin{theorem}\label{rmkB}
Let $\mathscr{C}$, $T$, and $U$ be as in Definition \ref{boilerplateDef}. Then $\mathscr{T}$ and $\mathscr{U}$ have the same number of indecomposable objects.
\end{theorem}

The above two results will be proven in general; that is, no mutability is required. Finally, we will prove the following:
\begin{theorem}\label{TheoremC}
Let $\mathscr{C}$, $T$, and $U$ be as in Definition \ref{boilerplateDef}. Suppose that $d$ is odd, and that $U$ is mutable at $u$ with the mutation $u^*$ such that $u$ and $u^*$ form an exchange pair. Then either (i) or (ii) below is true.
\begin{itemize}
\item[(i)] $\cvecU([t]) \geq 0$ for all $t \in \mathscr{T}$ and
\begin{align*}
\cvecU{}([t]) = \textrm{dim}_K \textrm{Hom}_{\Gamma_T}(F_T(t), \textrm{Im}(\delta_*))
\end{align*}
where $\delta$ is the morphism from $u$ to $\Sigma^d u^*$ shown in equation (\ref{exchangeOne}) and $\delta_* = F_T(\delta)$.
\item[(ii)]$\cvecU \leq 0$ for all $t \in \mathscr{T}$ and
\begin{align*}
\cvecU{}([t]) = -\textrm{dim}_K\textrm{Hom}_{\Gamma_T}(F_T(t), \textrm{Im}(\epsilon_*))
\end{align*}
where $\epsilon$ is the morphism from $u^*$ to $\Sigma^d u$ shown in equation (\ref{exchangeTwo}) and $\epsilon_* = F_T(\epsilon)$.
\end{itemize}
%Then there exists a map $u \xrightarrow{\delta} \Sigma^d u^*$ such that
%\begin{align*}
%\textrm{Hom}_{\Gamma_T}(F_T(t), \textrm{Im} \delta_*) = \textrm{Im}(\homs{}(t, \delta)).
%\end{align*}
%This means that the entries in a $c$-vector are given by dimension vectors of modules in $\textrm{Mod}\Gamma_T$.
\end{theorem}

Note that if $t$ is an indecomposable summand of $T$ then $F_T(t)$ is an indecomposable projective $\Gamma_T$-module. Hence $\textrm{dim}_K\textrm{Hom}_{\Gamma_T}(F_T(t), M)$ is an entry in the dimension vector of $M$ when $M \in \textrm{mod}\Gamma_T$ and Theorem \ref{TheoremC} shows that certain sign coherent $c$-vectors can be realised as dimension vectors.

\endgroup
\section{Definitions}
We begin with some definitions. For the purpose of this paper, $K$ is an algebraically closed field. We note also that by $\textrm{mod }\Lambda$ we denote the right $\Lambda$-modules for a finite dimensional $K$-algebra $\Lambda$.
\begin{definition}\cite[Definition~2.1]{GKO}
Let $\mathscr{C}$ be an additive category with an automorphism $\Sigma^d$ for $d \in \mathbb{Z}, 0<d$. The inverse is denoted $\Sigma^{-d}$, but we note that $\Sigma^d$ is not assumed to be the $d$-th power of another functor. Then a \textit{$\Sigma^d$-sequence} in $\mathscr{C}$ is a diagram of the form
\begin{align}
c^0 \xrightarrow{\gamma^0} c^1 \rightarrow c^2 \rightarrow \cdots \rightarrow c^d \rightarrow c^{d+1} \xrightarrow{\gamma^{d+1}} \Sigma^d(c^0).
\end{align}
\end{definition}

\begin{definition}\cite[Definition~2.1]{GKO}
A \textit{$(d+2)$-angulated category} is a triple ($\mathscr{C}, \Sigma^d, \pentagon$) where $\pentagon$ is a class of $\Sigma^d$-sequences called $(d+2)$-angles, satisfying the following conditions:
\begin{itemize}
\item [(N1)] $\pentagon$ is closed under sums and summands, and contains the $(d+2)$-angle 
\begin{align*}
c \xrightarrow{id_c} c \rightarrow 0 \rightarrow \cdots \rightarrow 0 \rightarrow 0 \rightarrow \Sigma^d(c)
\end{align*}
for each $c \in \mathscr{C}$. For each morphism $c^0 \xrightarrow{\gamma^0} c^1$ in $\mathscr{C}$, the class $\pentagon$ contains a $\Sigma^d$-sequence of the form in Definition 1.1.
\item [(N2)] The $\Sigma^d$-sequence (1) is in $\pentagon$ if and only if the $\Sigma^d$-sequence
\begin{align*}
c^1 \xrightarrow{\gamma^1} c^2 \rightarrow c^3 \rightarrow \cdots \rightarrow c^{d+1} \rightarrow \Sigma^d(c^0) \xrightarrow{(-1)^d\Sigma^d(\gamma^0)} \Sigma^d(c^1)
\end{align*}
is in $\pentagon$. This sequence is known as the left rotation of sequence (1).
\item [(N3)] A commutative diagram with rows in $\pentagon$ has the following extension property: \\
\begin{tikzpicture}[line cap = round, line join = round]

\node (a) at (-3, 2) {$b^0$};
\node (b) at (-1, 2) {$b^1$};
\node (c) at (1, 2) {$b^2$};
\node (d) at (3, 2) {$\ldots$};
\node (e) at (5, 2) {$b^{d}$};
\node (f) at (7, 2) {$b^{d+1}$};
\node (g) at (9, 2) {$\Sigma^d(b^0)$};
\node (h) at (-3, 0) {$c^0$};
\node (i) at (-1, 0) {$c^1$};
\node (j) at (1, 0) {$c^2$};
\node (k) at (3, 0) {$\ldots$};
\node (l) at (5, 0) {$c^{d}$};
\node (m) at (7, 0) {$c^{d+1}$};
\node (n) at (9, 0) {$\Sigma^d(c^0)$};

%Top row
\draw [->] (a) to (b);
\draw [->] (b) to (c);
\draw [->] (c) to (d);
\draw [->] (d) to (e);
\draw [->] (e) to (f);
\draw [->] (f) to (g);

%Bottom row
\draw [->] (h) to (i);
\draw [->] (i) to (j);
\draw [->] (j) to (k);
\draw [->] (k) to (l);
\draw [->] (l) to (m);
\draw [->] (m) to (n);

%Between
\draw [->] (a) to node[pos=0.5, right] {$\beta_0$} (h);
\draw [->] (b) to (i);
\draw [dotted, ->] (c) to (j);
\draw [dotted, ->] (e) to (l);
\draw [dotted, ->] (f) to (m);
\draw [->] (g) to node[pos=0.5, right] {$\Sigma^d(\beta_0)$} (n);

\end{tikzpicture}
\item[(N4)] The Octahedral Axiom, see \cite[Definition~2.1]{GKO}.
\end{itemize}
\end{definition}

\begin{definition}
Let $\mathscr{C}$ be a $(d+2)$-angulated category, and let $D=\textrm{Hom}_K(-, K)$ be the usual duality functor. A \textit{Serre functor} for $\mathscr{C}$ is an auto-equivalence $S: \mathscr{C} \to \mathscr{C}$ together with a family of isomorphisms which are natural in $X$ and $Y$
\begin{align*}
t_{X, Y}:\homs{}(Y, SX) \to D\homs{}(X, Y).
\end{align*}
We call the category $\mathscr{C}$ \textit{$2d$-Calabi-Yau} if $\mathscr{C}$ admits a Serre functor which is isomorphic to $(\Sigma^d)^2$, which we often write $\Sigma^{2d}$.
\end{definition}

We also state here a result that will be instrumental. Recall that if we let $\mathscr{C}$ be a $(d+2)$-angulated category, and let $T$ be an Oppermann-Thomas cluster tilting object with $\Gamma_T = \textrm{End}(T)$, then we have the functor $F_T: \mathscr{C} \to \textrm{mod }\Gamma_T$ that acts by sending $x \in \mathscr{C}$ to $\homs{}(T, x)$. In fact, by \cite[Theorem~0.5]{JorgJac}, we have a commutative diagram \\
\begin{center}
\begin{tikzpicture}[line cap = round, line join = round]

\node (a) at (-2, 2) {$\mathscr{C}$};
\node (b) at (2, 2) {$\textrm{mod }\Gamma_T$};
\node (c) at (-2, 0) {$\frac{\mathscr{C}}{[\Sigma^d \mathscr{T}]}$};
\node (d) at (2, 0) {$\mathscr{D}$};

%Top row
\draw [->] (a) to node[pos=0.5, above] {$F_T$} (b);
\draw [->] (a) to (c);
\draw [->] (c) to node[pos=0.5, above] {$ \sim $} (d);
\draw [right hook->] (d) to (b);

\end{tikzpicture}
\end{center}
where $\mathscr{D}$ is $d$-cluster tilting in $\textrm{mod }\Gamma_T$. This means that for $x, y \in \mathscr{C}$, we have that $\tilthoms{}(x, y) \cong \textrm{Hom}_{\Gamma_T}(F_T(x), F_T(y))$. Here $[\Sigma^d \mathscr{T}]$ is the ideal of $\mathscr{C}$ consisting of the morphisms that factor through $\Sigma^d \mathscr{T}$.

Using this definition, the result is as follows:
\begin{theorem}\cite[Theorem~C]{JustJorg}\label{thmC}
Let $\mathscr{C}$ be a $2d$-Calabi-Yau $(d+2)$-angulated category that is $K$-linear, Hom-finite, and Krull-Schmidt. Let $T$ be an Oppermann-Thomas cluster tilting object of $\mathscr{C}$. Then there is a homomorphism of abelian groups $\theta: K_0(\textrm{mod } \Lambda) \to \splitG{}(\mathscr{T})$ such that for any $(d+2)$-angle
\begin{align*}
s_{d+1} \to s_d \to \ldots \to s_0 \xrightarrow{\gamma} \Sigma^ds_{d+1}
\end{align*}
in $\mathscr{C}$, we have that
\begin{align*}
\Sigma_{i=0}^{d+1}(-1)^i \textrm{Ind}_{\mathscr{T}}(s_i) = \theta([\textrm{Im }F_T\gamma]).
\end{align*}
\end{theorem}
\section{Tropical Duality}
\subsection{Proof of the duality}
Firstly, we would like to extend our definition of index. The split Grothendieck group of $\mathscr{C}$ can be defined in the same way as for $\mathscr{T}$. Then we may define a homomorphism
\begin{align*}
\indT{}: \splitG{}(\mathscr{C}) \to \splitG{}(\mathscr{T})
\end{align*}
by
\begin{align*}
\indT{}([c]) := \indT{}(c)
\end{align*}
for all $c \in \mathscr{C}$. We also note that the translation functor maps the split Grothendieck group of $\mathscr{C}$ to itself, in the following way:
\begin{align*}
\Sigma^d : &\splitG{}(\mathscr{C}) \to \splitG{}(\mathscr{C}). \\
&[x] \to [\Sigma^d x]
\end{align*}

We now prove Theorem \ref{ThmA}, which we restate here:
\begin{theorem}\label{tropDual}
Let $\mathscr{C}$, $T$, and $U$ be as in Definition \ref{boilerplateDef}. Then there are inverse isomorphisms
\begin{center}
\begin{tikzpicture}[line cap = round, line join = round]
\node (a) at (-3, 0) {$K_0^{\textrm{split}}(\mathscr{T})$};
\node (b) at (3, 0) {$K_0^{\textrm{split}}(\mathscr{U}).$};

%\draw ([yshift=0.8 cm]a.east) -- (b);
\draw [->] ([yshift= 0.2 cm]a.east) to node[above]{$(-1)^d\textrm{ind}_{\mathscr{U}} \circ \Sigma^d|_{K_0^{\textrm{split}}(\mathscr{T})}$} ([yshift= 0.2 cm]b.west);
\draw [->] ([yshift= -0.2 cm]b.west) to node[below]{$\textrm{ind}_{\mathscr{T}}|_{K_0^{\textrm{split}}(\mathscr{U})}$} ([yshift= -0.2 cm]a.east);
\end{tikzpicture}
\end{center}
\end{theorem}

\begin{proof}
Let $u \in \mathscr{U}$ be given. By Definition \ref{OTDef}, there is a $(d+2)$-angle
\begin{align*}
t_d \to t_{d-1} \to \ldots \to t_0 \to u \to \Sigma^dt_d
\end{align*}
with each $t_i \in \mathscr{T}$. Then $\indT{}([u]) = \Sigma_{i=0}^d (-1)^i[t_i]$, and so
\begin{align}\label{eqn1}
\Sigma^d \circ \indT{}([u]) = \Sigma_{i=0}^d (-1)^i[\Sigma^dt_i].
\end{align}\\

By rotating this $(d+2)$-angle, we also have the $(d+2)$-angle
\begin{align*}
u \to \Sigma^d t_d \to \ldots \to \Sigma^d t_0 \xrightarrow{\delta} \Sigma^d u.
\end{align*}

By Theorem \ref{thmC} we have that 
\begin{align*}
(-1)^{d+1}\indU{}([u]) + \Sigma_{i=0}^d (-1)^i\indU{}([\Sigma^d t_i]) = \theta([\textrm{Im }F_U(\delta)]).
\end{align*}
Notice that $F_U(\delta)$ is a map from $U$ to $\Sigma^d u$, which by the definition of Oppermann-Thomas cluster tilting objects (definition \ref{OTDef}) is zero. We obtain that 
\begin{align*}
(-1)^{d+1}\indU{}([u]) = - \Sigma_{i=0}^d (-1)^i\indU{}([\Sigma^d t_i])
\end{align*}
which gives us that
\begin{align}\label{eqn2}
\indU{}([u]) = (-1)^d \Sigma_{i=0}^d (-1)^i\indU{}([\Sigma^d t_i])
\end{align}
as $(-1)^{d+2} = (-1)^d$. We have shown that 
\begin{align*}
[u] &= \indU{}([u]) \\
&\stackrel{\mathclap{\normalfont\mbox{(\ref{eqn2})}}}{=}(-1)^d \Sigma_{i=0}^d (-1)^i\indU{}([\Sigma^d t_i]) \\
&= (-1)^d \indU{}(\Sigma_{i=0}^d (-1)^i[\Sigma^d t_i]) \\
&\stackrel{\mathclap{\normalfont\mbox{(\ref{eqn1})}}}{=} (-1)^d \indU{}(\Sigma^d \circ \indT{}([u])) \\
&= (-1)^d \indU{}\circ \Sigma^d \circ \indT{}([u]).
\end{align*}

Proceeding in a similar fashion, let $t \in \mathscr{T}$ be given. Again by definition we have a $(d+2)$-angle
\begin{align*}
u_d \to u_{d-1} \to \ldots \to u_0 \to \Sigma ^d t \to \Sigma^du_d
\end{align*}
with each $u_i \in \mathscr{U}$. Then we have the $(d+2)$-angle
\begin{align*}
t \to u_d \to u_{d-1} \to \ldots \to u_0 \to \Sigma ^d t.
\end{align*}
The first $(d+2)$-angle gives us that
\begin{align*}
\indU{}([\Sigma^dt]) = \Sigma_{i=0}^d (-1)^i[u_i],
\end{align*}
and the second gives us that
\begin{align*}
\Sigma_{i=0}^d (-1)^i\indT{}(u_i) + (-1)^{d+1}\indT(t) = 0
\end{align*}
by Theorem \ref{thmC}. This means that
\begin{align*}
\indT{}(\indU{}([\Sigma^dt])) + (-1)^{d+1}\indT{}(t) = 0.
\end{align*}
Written another way this is
\begin{align*}
[t] = \indT{}([t]) = (-1)^d \indT{} \circ \indU \circ \Sigma^d ([t])
\end{align*}
as required.
%By rotating the $(d+2)$-angle, we have the following:
%\begin{align*}
%t \to \Sigma^du_d \to \ldots \to \Sigma^du_0 \xrightarrow{\epsilon} \Sigma^dt.
%\end{align*}
%By Theorem \ref{thmC} there is a homomorphism of groups $\theta: K_0(\textrm{mod } \Gamma_T) \to \splitG{}(\mathscr{T})$ where $\Gamma_T = \textrm{End}(T)$ such that 
%\begin{align*}
%(-1)^{d+1}\indT{}([t]) + \Sigma_{i=0}^d (-1)^i\indT{}([\Sigma^d u_i]) = \theta([\textrm{Im }F_T(\epsilon)]).
%\end{align*}
%By definition \ref{OTDef}, $F_T(\epsilon) = 0$, so we obtain
%\begin{align*}
%\indT{}([t]) = (-1)^d \Sigma_{i=0}^d (-1)^i\indT{}([\Sigma^d u_i])
%\end{align*}
%in the same fashion as above.
%
%Using these facts together we see that
%\begin{align*}
%[t] &= \indT{}([t]) \\
%&=(-1)^d \Sigma_{i=0}^d (-1)^i\indT{}([\Sigma^d u_i]) \\
%&= \indT{}( (-1)^d \Sigma_{i=0}^d (-1)^i[\Sigma^d u_i]) \\
%&= \indT{}((-1)^d\indU{} \circ \Sigma^d([t])) \\
%&= \indT{} \circ (-1)^d \indU{} \circ \Sigma^d([t])
%\end{align*}
%Thus we have shown that $\indT{} \circ (-1)^d \indU{} \circ \Sigma^d = \textrm{Id}_{\splitG{}(\mathscr{T})}$ and that $(-1)^d \indU{} \circ \Sigma^d \circ \indT{} = \textrm{Id}_{\splitG{}(\mathscr{U})}$. This shows that the homomorphisms are mutually inverse, as required.
\end{proof}
\subsection{Immediate Consequences}
We see that Theorem \ref{rmkB} follows immediately from Theorem \ref{tropDual}.
We also have the following immediate consequence of Theorem \ref{tropDual}:

\begin{theorem}\label{dknew}
Let $\mathscr{C}$, $T$, and $U$ be as in Definition \ref{boilerplateDef}. Then
\begin{align*}
\splitG{}(\mathscr{T}) = \bigoplus_{u \in \indec{}(\mathscr{U})} \mathbb{Z} \cdot \indT{}(u).
\end{align*}
\end{theorem}
\section{Categorical $c$-vectors}
\subsection{Using Tropical Duality with Categorical $c$-vectors}
We may use Theorem \ref{tropDual} to show two formulae for the computation of $c$-vectors.
\begin{theorem}\label{lemmap1}
Let $\mathscr{C}$, $T$, and $U$ be as in Definition \ref{boilerplateDef}. Then the $c$-vector of the pair $(u, \mathscr{U})$ with respect to the Oppermann-Thomas cluster tilting subcategory $\mathscr{T}$ is given by
\begin{align*}
\cvecU{} = (-1)^d [u]^* \circ \indU{} \circ \Sigma^d |_{\splitG{}(\mathscr{T})}.
\end{align*}
\end{theorem}
\begin{proof}
Let $v \in \indec{}(\mathscr{U})$ be given. Then
\begin{align*}
(-1)^d[u]^* \circ \indU{} \circ \Sigma^d(\indT{}(v)) &= [u]^* \circ (-1)^d \indU{} \circ \Sigma^d \circ \indT{}(v) \\
&= [u]^*(v) \textrm{ by Theorem 3.1} \\
&= \delta_{uv} \\
&= \cvecU{}(\indT{}(v)).
\end{align*}
\end{proof}

\begin{lemma}\label{simpleLemma}
Let $\mathscr{C}$, $T$, and $U$ be as in Definition \ref{boilerplateDef}. Suppose that $U$ is mutable at $u$ with the mutation $u^*$ such that $u$ and $u^*$ form an exchange pair. Then $F_U(\Sigma^d u^*)$ is simple in the abelian category $\textrm{mod }\Gamma_U$.
\end{lemma}
\begin{proof}
We have
\begin{align*}
\homs{}(U, \Sigma^d u^*) = \homs{}(u, \Sigma^d u^*)
\end{align*}
by definition \ref{OTDef}(i). By the assumption that $u$ and $u^*$ form an exchange pair, the space $\homs{}(u, \Sigma^d u^*)$ is one dimensional. This means that in $\textrm{mod }\Gamma_U$, the object $\homs{}(U, \Sigma^d u^*) = F_U(x)$ is one dimensional. This gives us the simplicity of the object.
\end{proof}

We may immediately use this simplicity to prove another lemma:
\begin{lemma}\label{zeroLemma}
Let $\mathscr{C}$, $T$, and $U$ be as in Definition \ref{boilerplateDef}. Suppose that $U$ is mutable at $u$ with the mutation $u^*$ such that $u$ and $u^*$ form an exchange pair. Let $t, t'$ be (not necessarily distinct) indecomposable objects of $\mathscr{T}$. Then for any $n \in \mathbb{Z}$, at least one of the homomorphism spaces
\begin{align*}
\tilthomsU{}(\Sigma^{nd} t, \Sigma^d u^*)
\end{align*}
and
\begin{align*}
\tilthomsU{}(\Sigma^d u^*, \Sigma^{(n+1)d} t')
\end{align*}
is zero.
\end{lemma}
\begin{proof}
Suppose that there is a non-zero morphism in $\tilthomsU{}(\Sigma^{nd} t, \Sigma^d u^*)$, and a non-zero morphism in $\tilthomsU{}(\Sigma^d u^*, \Sigma^{(n+1)d} t')$. We recall that
\begin{align*}
\tilthomsU{}(\Sigma^{nd} t, \Sigma^d u^*) \cong \textrm{Hom}_{\textrm{mod }\Gamma_U}(F_U(\Sigma^{nd} t), F_U(\Sigma^d u^*)),
\end{align*}
and that
\begin{align*}
\tilthomsU{}(\Sigma^du^*, \Sigma^{(n+1)d}t') \cong \textrm{Hom}_{\textrm{mod }\Gamma_U}(F_U(\Sigma^du^*), F_U(\Sigma^{(n+1)d}t')).
\end{align*}
This means that there is a non-zero morphism in $\textrm{Hom}_{\textrm{mod }\Gamma_U}(F_U(\Sigma^{nd} t), F_U(\Sigma^d u^*))$, and a non-zero morphism in $\textrm{Hom}_{\textrm{mod }\Gamma_U}(F_U(\Sigma^du^*), F_U(\Sigma^{(n+1)d}t'))$. By lemma \ref{simpleLemma} the object $F_U(\Sigma^du^*)$ is simple, so these two morphisms must compose to a non-zero morphism from $F_U(\Sigma^{nd} t)$ to $F_U(\Sigma^{(n+1)d}t')$. Again by the above isomorphisms, this composed morphism means we have a non-zero morphism in $\frac{\mathscr{C}}{[\Sigma^d U]}$ from $\Sigma^{nd} t$ to $\Sigma^{(n+1)d} t'$, hence also a non-zero morphism in $\mathscr{C}$ from $\Sigma^{nd} t$ to $\Sigma^{(n+1)d} t'$. This is a contradiction of the fact that $T$ is an Oppermann-Thomas cluster tilting object. This proves the lemma.
\end{proof}

\begin{lemma}\label{myLemma}
Let $\mathscr{C}$, $T$, and $U$ be as in Definition \ref{boilerplateDef}. Suppose also that $d$ is odd, and that $U$ is mutable at $u$ with the mutation $u^*$ such that $u$ and $u^*$ form an exchange pair. Then for $t \in \mathscr{T}$,
\begin{align*}
\cvecU{}([t])= (-1)^d[\textrm{dim}_K\tilthomsU{}(\Sigma^d t, \Sigma^d u^*) + (-1)^d\textrm{dim}_K\tilthomsU{}(\Sigma^du^*, \Sigma^{2d}t)].
\end{align*}
Moreover, at least one of $\tilthomsU{}(\Sigma^d t, \Sigma^d u^*)$ and $\tilthomsU{}(\Sigma^du^*, \Sigma^{2d}t)$ is zero.
\end{lemma}
\begin{proof}
By the definition of an Oppermann-Thomas cluster tilting object, there is a $(d+2)$-angle
\begin{align*}
u_d \to u_{d-1} \to \ldots \to u_0 \to \Sigma^d t \to \Sigma^d u_d,
\end{align*}
with each $u_i \in \mathscr{U}$. Then $\indU{}(\Sigma^dt) = \Sigma_{i=0}^d(-1)^i[u_i]$. For each $u_i$, we see that $u_i = u^{\beta_i} \oplus \tilde{u}_i$, where $u$ is not a direct summand of $\tilde{u}_i$. Then we have that $[u]^*([u_i])$ is equal to the number of copies of $u$ in this sum; that is $[u]^*([u_i]) = \beta_i$. We also have that
\begin{align*}
\textrm{dim}_K\homs{}(u_i, \Sigma^d u^*) &=\textrm{dim}_K\homs{}(u^{\beta_i} \oplus \tilde{u}_i, \Sigma^d u^*) \\
&= \textrm{dim}_K\homs{}(u^{\beta_i}, \Sigma^d u^*) \\
&=\beta_i * \textrm{dim}_K\homs{}(u, \Sigma^d u^*) \\
&=\beta_i,
\end{align*}
because for each indecomposable $u_\alpha$ in $\mathscr{U}$ not equal to $u$ we have that $\homs{}(u_\alpha, \Sigma^d u^*) = 0$ and by definition \ref{exchangePair} we have that $\textrm{dim}_K\homs{}(u, \Sigma^d u^*) = 1$. This means we have that $[u]^*([u_i]) = \textrm{dim}_K\homs{}(u_i, \Sigma^d u^*)$.

We apply $\cvecU{}$ to $t$ and obtain
\begin{align*}
\cvecU{}([t]) &= (-1)^d [u]^* \circ \indU{} \circ \Sigma^d ([t]) \textrm{ by Lemma \ref{lemmap1}} \\
&= (-1)^d[u]^* \circ \indU{} ([\Sigma^d t]) \\
&= [u]^* ( (-1)^d \Sigma_{i=0}^d(-1)^i[u_i]) \\
&= (-1)^d \Sigma_{i=0}^d(-1)^i[u]^*([u_i]) \\
&= (-1)^d \Sigma_{i=0}^d(-1)^i\textrm{dim}_K\homs{}(u_i, \Sigma^d u^*) \\
&= (-1)^d[\textrm{dim}_K\tilthomsU{}(\Sigma^d t, \Sigma^d u^*) + (-1)^d\textrm{dim}_K\tilthomsU{}(\Sigma^du^*, \Sigma^{2d}t)].
\end{align*}
The last step here is by \cite[Proposition~3.1]{Reid}. \\

By lemma \ref{zeroLemma} at least one of these hom-spaces is zero, and we have proven the lemma.
%We may apply the functor $F_U$ here and see that if $\tilthomsU{}(\Sigma^d t, \Sigma^d u^*)$ and
% $\tilthomsU{}(\Sigma^du^*, \Sigma^{2d}t)$ both contain a non-zero map, then so too do both $\tilthomsU{}(F_U(\Sigma^d t), F_U(\Sigma^d u^*))$ and
%  $\tilthomsU{}(F_U(\Sigma^du^*), F_U(\Sigma^{2d}t))$. Then by lemma \ref{simpleLemma}, we have a surjective map from $F_U(\Sigma^d t)$ to $F_U(\Sigma^d u^*)$ and an injective map from $F_U(\Sigma^d u^*)$ to $F_U(\Sigma^{2d}t)$. This gives us a non-zero map from $F_U(\Sigma^d t)$ to $F_U(\Sigma^{2d}t)$, which by the isomorphism proven in \cite[Theorem~0.5]{JorgJac} means that we have a non-zero map from $\Sigma^d t$ to $\Sigma^{2d}t$. This contradicts the definition of $\mathscr{T}$, so at least one of $\tilthomsU{}(\Sigma^d t, \Sigma^d u^*)$ and $\tilthomsU{}(\Sigma^du^*, \Sigma^{2d}t)$ must be zero. This proves the lemma.
\end{proof}

Lemma \ref{myLemma} allows us to show a sign coherence property of the $c$-vector:
\begin{lemma}\label{coherenceLemma}
Let $\mathscr{C}$, $T$, and $U$ be as in Definition \ref{boilerplateDef}. Suppose also that $d$ is odd, and that $U$ is mutable at $u$ with the mutation $u^*$ such that $u$ and $u^*$ form an exchange pair. Then either $\cvecU{}([t]) \geq 0$ for all $t \in \mathscr{T}$, or $\cvecU{}([t]) \leq 0$ for all $t \in \mathscr{T}$.
\end{lemma}
\begin{proof}
Suppose that there exist $t^1, t^2 \in \mathscr{T}$ such that $\cvecU{}([t^1]) > 0$ and $\cvecU{}([t^2]) < 0$. By the definition of an Oppermann-Thomas cluster tilting subcategory, there are two $(d+2)$-angles
\begin{align*}
u_d^1 \to u_{d-1}^1 \to \ldots \to u_0^1 \to \Sigma^d t^1 \to \Sigma^d u_d^1
\end{align*}
and
\begin{align*}
u_d^2 \to u_{d-1}^2 \to \ldots \to u_0^2 \to \Sigma^d t^2 \to \Sigma^d u_d^2,
\end{align*}
where each $u_i^j \in \mathscr{U}$.
This means that $\indU{}([\Sigma^d t^j]) = \Sigma_{i=0}^d(-1)^i[u_i^j]$. By Lemma \ref{myLemma}, we have that
\begin{align*}
\cvecU{}([t^j]) = (-1)^d[\textrm{dim}_K\tilthomsU{}(\Sigma^d t^j, \Sigma^d u^*) + (-1)^d\textrm{dim}_K\tilthomsU{}(\Sigma^du^*, \Sigma^{2d}t^j)],
\end{align*}
where at least one of the terms on the right is zero. Then as we have assumed $d$ is odd, and we know that the dimension of a space is always non-negative, we see that
\begin{align*}
\cvecU{}([t^1]) = \textrm{dim}_K\tilthomsU{}(\Sigma^du^*, \Sigma^{2d}t^1)
\end{align*}
and
\begin{align*}
\cvecU{}([t^2]) = -\textrm{dim}_K\tilthomsU{}(\Sigma^d t^2, \Sigma^d u^*).
\end{align*}
This immediately gives us a contradiction by lemma \ref{zeroLemma}, and as such our initial assumption must be false. This proves the lemma.
\end{proof}

\begin{lemma}\label{cokerLemma}
Let $\mathscr{C}$, $T$, and $U$ be as in Definition \ref{boilerplateDef}. Suppose that $d$ is odd, and suppose that $U$ is mutable at $u$ with the mutation $u^*$ such that $u$ and $u^*$ form an exchange pair. Let $\phi$ be the morphism from $\Sigma^d u^*$ to $\Sigma^d e_d$ which comes from rotating the exchange $(d+2)$-angle shown in equation (\ref{exchangeOne}). For any object $z \in \mathscr{C}$, the morphism $\phi$ induces the morphism $\phi^*: \homs{}(\Sigma^d e_d, z) \to \homs{}(\Sigma^d u^*, z)$. Then the cokernel of $\phi^*$ is $\tilthomsU{}(\Sigma^d u^*, z)$.
\end{lemma}

\begin{proof}
We have the rotated $(d+2)$-angle
\begin{align*}
u \xrightarrow{\delta}\Sigma^d u^* \xrightarrow{\phi} \Sigma^d e_d \to \ldots \to \Sigma^d e_1 \to \Sigma^d u
\end{align*}
which gives a long exact sequence
\begin{align*}
\homs{}(\Sigma^d u, z) \to \ldots \to \homs{}(\Sigma^d e_d, z) \xrightarrow{\phi^*} \homs{}(\Sigma^d u^*, z) \xrightarrow{\delta^*} \homs{}(u, z).
\end{align*}
%Changes start here
We can obtain from this an exact sequence
\begin{align*}
\homs{}(\Sigma^d e_d, z) \xrightarrow{\phi^*} \homs{}(\Sigma^d u^*, z) \to \textrm{Coker}(\phi^*) \to 0.
\end{align*}
It is enough to prove that $\textrm{Im}(\phi^*)$ is equal to
 $\homs{}^{[\Sigma^d\mathscr{U}]}(\Sigma^d u^*, z)$, where $\homs{}^{[\Sigma^d \mathscr{U}]}(\Sigma^d u^*, z)$ denotes the morphisms from $\Sigma^d u^*$ to $z$ that factor through $[\Sigma^d \mathscr{U}]$. Moreover, we have $\textrm{Im}(\phi^*) = \textrm{Ker}(\delta^*)$.

Firstly, take an element $\theta \in \homs{}(\Sigma^d u^*, z)$ such that $\theta$ factors through $\Sigma^d \mathscr{U}$. Then $\delta^*(\theta) = 0$, or we would have a non-zero morphism from $\mathscr{U}$ to $\Sigma^d \mathscr{U}$. So $\homs{}^{[\Sigma^d \mathscr{U}]}(\Sigma^d u^*, z) \subseteq \textrm{Ker}(\delta^*)$. Then if we have $\epsilon \in  \homs{}(\Sigma^d u^*, z)$ such that $\delta^*(\epsilon) = 0$, by exactness $\epsilon$ factors through $\Sigma^d e_d \in \Sigma^d \mathscr{U}$. Thus, $\homs{}^{[\Sigma^d \mathscr{U}]}(\Sigma^d u^*, z) = \textrm{Ker}(\delta^*)$.
%Changes end here
%In order to see that the statement is true, it is enough to prove that $\textrm{Im}(\phi^*) \cong \homs{}^{[\Sigma^d\mathscr{U}]}(\Sigma^d u^*, z)$, where $\homs{}^{[\Sigma^d \mathscr{U}]}(\Sigma^d u^*, z)$ denotes the morphisms from $\Sigma^d u^*$ to $z$ that factor through $[\Sigma^d \mathscr{U}]$. We examine the image of $\phi^*$. \\
%
%Firstly, take $\theta \in \textrm{Im}(\phi^*)$. Then $\theta = \theta' \circ \phi$, where $\theta'$ is a morphism from  $\Sigma^d e_d$ to $z$. Then obviously $\theta \in \homs{}^{[\Sigma^d \mathscr{U}]}(\Sigma^d u^*, z)$. Thus $\textrm{Im}(\phi*) \subseteq \homs{}^{[\Sigma^d \mathscr{U}]}(\Sigma^d u^*, z)$.\\
%
%Now, take $\gamma \in \homs{}^{[\Sigma^d \mathscr{U}]}(\Sigma^d u^*, z)$. Then $\gamma =\gamma'' \circ \gamma'$ for $\gamma': \Sigma^d u^* \to \Sigma^d u'$ and $\gamma'':\Sigma^d u' \to z$, where $u' \in \mathscr{U}$. Then obviously $\gamma' \circ \delta = 0$, so $\gamma \circ \delta = 0$ and we have that $\gamma \in \textrm{Ker}(\delta^*)$. By exactness of the above sequence, we have that $\gamma \in \textrm{Im}(\phi^*)$. This gives us that $\homs{}^{[\Sigma^d \mathscr{U}]}(\Sigma^d u^*, z) \subseteq \textrm{Im}(\phi*)$. We have shown then that $\homs{}^{[\Sigma^d \mathscr{U}]}(\Sigma^d u^*, z) \cong \textrm{Im}(\phi^*)$. Then we have that the cokernel of $\phi^*$ is isomorphic to $\homs{}(\Sigma^d u^*, z)/\homs{}^{[\Sigma^d \mathscr{U}]}(\Sigma^d u^*, z)$, which is isomorphic to $\tilthomsU{}(\Sigma^d u^*, z)$. This proves the result.
\end{proof}

We may use these results to prove more properties of the $c$-vector.

\begin{proposition}\label{someProp}
Let $\mathscr{C}$, $T$, and $U$ be as in Definition \ref{boilerplateDef}. Suppose that $d$ is odd, and suppose that $U$ is mutable at $u$ with the mutation $u^*$ such that $u$ and $u^*$ form an exchange pair.
\begin{itemize}
\item[(i)] If $\cvecU([t]) \geq 0$ for all $t \in \mathscr{T}$, then for any $t \in \mathscr{T}$
\begin{align*}
\cvecU{}([t]) = \textrm{dim}_K\textrm{Im}(\homs{}(t, u) \xrightarrow{\delta_*} \homs{}(t, \Sigma^d u^*))
\end{align*}
where $\delta$ is the morphism from $u$ to $\Sigma^d u^*$ shown in the exchange $(d+2)$-angle, as seen in equation (\ref{exchangeOne}).
\item[(ii)] If $\cvecU \leq 0$ for all $t \in \mathscr{T}$, then for any $t \in \mathscr{T}$
\begin{align*}
\cvecU{}([t]) = -\textrm{dim}_K\textrm{Im}(\homs{}(t, u^*) \xrightarrow{\epsilon_*} \homs{}(t, \Sigma^d u))
\end{align*}
where $\epsilon$ is the morphism from $u^*$ to $\Sigma^d u$ shown in the exchange $(d+2)$-angle, as seen in equation (\ref{exchangeTwo}).
\end{itemize}
\end{proposition}
\begin{proof}
Firstly, let $t \in \mathscr{T}$ be given. Then, as $\mathscr{U}$ is Oppermann-Thomas cluster tilting, there is a $(d+2)$-angle
\begin{align*}
u_d \to u_{d-1} \to \ldots \to u_0 \to \Sigma^d t \to \Sigma^d u_d.
\end{align*}
By Lemma \ref{myLemma}, we have that
\begin{align}\label{importantEquation}
\cvecU{}([t]) &=  (-1)^d[\textrm{dim}_K\tilthomsU{}(\Sigma^d t, \Sigma^d u^*) + (-1)^d\textrm{dim}_K\tilthomsU{}(\Sigma^du^*, \Sigma^{2d}t)],
\end{align}
where at least one of the terms on the right hand side is zero.

\paragraph{Part (i):}
As we have assumed that $\cvecU{}([t])$ is non-negative and $d$ is odd, we have that
\begin{align*}
\cvecU{}([t]) &= \textrm{dim}_K\tilthomsU{}(\Sigma^du^*, \Sigma^{2d}t).
\end{align*}
We also have the exchange $(d+2)$-angle
\begin{align*}
u^* \to e_d \to \ldots \to e_1 \to u \xrightarrow{\delta} \Sigma^d u^*.
\end{align*}
This induces a morphism
\begin{align*}
\delta_*:\homs{}(t, u) \to \homs{}(t, \Sigma^d u^*).
\end{align*}
Then by the above, our claim becomes:
For $t \in \mathscr{T}$,
\begin{align*}
\textrm{dim}_K\textrm{Im}(\delta_*) = \textrm{dim}_K\tilthomsU{}(\Sigma^du^*, \Sigma^{2d}t).
\end{align*}

We have a long exact sequence
\begin{align*}
\ldots \to  \homs{}(t, u) \xrightarrow{\delta_*} \homs{}(t, \Sigma^d u^*) \to \homs{}(t, \Sigma^d e_d) \to \ldots,
\end{align*}
from which we obtain an exact sequence
\begin{align*}
0 \to \textrm{Im}(\delta_*) \to \homs{}(t, \Sigma^d u^*) \to \homs{}(t, \Sigma^d e_d).
\end{align*}
As the category $\mathscr{C}$ is $2d$-Calabi-Yau, we may apply the Serre duality to this sequence to obtain the exact sequence
\begin{align*}
0 \to \textrm{Im}(\delta_*) \to D\homs{}(\Sigma^d u^*, \Sigma^{2d}t) \to D\homs{}(\Sigma^d e_d, \Sigma^{2d}t)
\end{align*}
to which we can apply the standard duality functor to obtain the exact sequence
\begin{align*}
\homs{}(\Sigma^d e_d, \Sigma^{2d}t) \xrightarrow{\phi^*} \homs{}(\Sigma^d u^*, \Sigma^{2d}t) \to D\textrm{Im}(\delta_*) \to 0,
\end{align*}
where $\phi:\Sigma^d u^* \to \Sigma^d e_d$. By lemma \ref{cokerLemma} we see that $D\textrm{Im}(\delta_*) \cong \tilthomsU{}(\Sigma^d u^*, \Sigma^{2d} t)$. This proves part (i).
\paragraph{Part (ii):}
As we have assumed that $\cvecU{}([t])$ is non-positive and $d$ is odd, we have from equation (\ref{importantEquation}) that
\begin{align*}
\cvecU{}([t]) &= -\textrm{dim}_K\tilthomsU{}(\Sigma^d t, \Sigma^d u^*).
\end{align*}
We also have the exchange $(d+2)$-angle
\begin{align*}
u \to f_d \to \ldots \to f_1 \to u^* \xrightarrow{\epsilon} \Sigma^d u.
\end{align*}
This induces a morphism
\begin{align*}
\epsilon_*:\homs{}(t, u^*) \to \homs{}(t, \Sigma^d u).
\end{align*}
As in the proof of part (i), we may rewrite the claim: for $t \in \mathscr{T}$,
\begin{align*}
\textrm{dim}_K\textrm{Im}(\epsilon_*) = \textrm{dim}_K\tilthomsU{}(\Sigma^dt, \Sigma^du^*).
\end{align*}
We have a long exact sequence
\begin{align*}
\ldots \to  \homs{}(t, u^*) \xrightarrow{\epsilon_*} \homs{}(t, \Sigma^d u) \to \homs{}(t, \Sigma^d f_d) \to \ldots,
\end{align*}
from which we obtain an exact sequence
\begin{align*}
0 \to \textrm{Im}(\epsilon_*) \to \homs{}(t, \Sigma^d u) \to \homs{}(t, \Sigma^d f_d).
\end{align*}
We wish to examine the image of $\epsilon_*$; in fact, we aim to prove that it is isomorphic to $\frac{\mathscr{C}}{[\mathscr{U}]}(t, u^*)$. Firstly, take an element $\theta \in \homs{}(t, u^*)$ such that $\theta$ factors through $\mathscr{U}$. Then $\epsilon_*(\theta) = 0$, or we would have a non zero morphism from $\mathscr{U}$ to $\Sigma^d\mathscr{U}$. So $\homs{}^{[\mathscr{U}]}(t,u^*) \subseteq \textrm{Ker}(\epsilon_*)$. Then, if we have $\phi \in \homs{}(t, u^*)$ such that $\epsilon_*(\phi) = 0$, by exactness we have that $\phi$ factors through $f_1 \in \mathscr{U}$. Thus, $\homs{}^{[\mathscr{U}]}(t,u^*) = \textrm{Ker}(\epsilon_*)$. As we know that $\textrm{Im}(\epsilon_*) \cong \frac{\homs{}(t, u^*)}{\textrm{Ker}(\epsilon_*)}$, this proves the result.
\end{proof}

\subsection{Proof of Theorem C}
We can use these results to prove our final claim. We restate Theorem \ref{TheoremC} here:
\begin{theorem*}
Let $\mathscr{C}$, $T$, and $U$ be as in Definition \ref{boilerplateDef}. Suppose that $d$ is odd, and suppose that $U$ is mutable at $u$ with the mutation $u^*$ such that $u$ and $u^*$ form an exchange pair. Then either (i) or (ii) below is true.
\begin{itemize}
\item[(i)] $\cvecU([t]) \geq 0$ for all $t \in \mathscr{T}$ and
\begin{align*}
\cvecU{}([t]) = \textrm{dim}_K \textrm{Hom}_{\Gamma_T}(F_T(t), \textrm{Im}(\delta_*))
\end{align*}
where $\delta$ is the morphism from $u$ to $\Sigma^d u^*$ in equation (\ref{exchangeOne}) and $\delta_* = F_T(\delta)$.
\item[(ii)] $\cvecU \leq 0$ for all $t \in \mathscr{T}$ and
\begin{align*}
\cvecU{}([t]) = -\textrm{dim}_K\textrm{Hom}_{\Gamma_T}(F_T(t), \textrm{Im}(\epsilon_*))
\end{align*}
where $\epsilon$ is the morphism from $u^*$ to $\Sigma^d u$  in equation (\ref{exchangeTwo}) and $\epsilon_* = F_T(\epsilon)$.
\end{itemize}
Note that if $t$ is an indecomposable summand of $T$ then $F_T(t)$ is an indecomposable projective $\Gamma_T$-module. Hence $\textrm{dim}_K\textrm{Hom}_{\Gamma_T}(F_T(t), M)$ is an entry in the dimension vector of $M$ when $M \in \textrm{mod}\Gamma_T$.
\end{theorem*}
\begin{proof}
Lemma \ref{coherenceLemma} says that either $\cvecU{}([t]) \geq 0$ for all $t \in \mathscr{T}$ or $\cvecU{}([t]) \leq 0$ for all $t \in \mathscr{T}$. Assume the former and begin with the map obtained from the exchange angle
\begin{align*}
u \xrightarrow{\delta} \Sigma^d u^*.
\end{align*}
We can apply the functor $F_T$ to obtain a commutative diagram
\begin{center}
\begin{tikzpicture}[line cap = round, line join = round]
\node (a) at (-3, 2) {$F_T(u)$};
\node (b) at (3, 2) {$F_T(\Sigma^d u^*).$};
\node (c) at (0,0) {$\textrm{Im}\delta_*$};

\draw [->] (a) to node[above]{$\delta_*$} (b);
\draw [->>] (a) to (c);
\draw[right hook->] (c) to (b);

\end{tikzpicture}
\end{center}
We may then apply the functor $\textrm{Hom}_{\Gamma_T}(F_T(t), -)$ to this diagram to obtain
\begin{center}
\begin{tikzpicture}[line cap = round, line join = round]
\node (a) at (-4, 2) {$\textrm{Hom}_{\Gamma_T}(F_T(t),F_T(u))$};
\node (b) at (4, 2) {$\textrm{Hom}_{\Gamma_T}(F_T(t),F_T(\Sigma^d u^*)).$};
\node (c) at (0,0) {$\textrm{Hom}_{\Gamma_T}(F_T(t),\textrm{Im}\delta_*)$};

\draw [->] (a) to node[above]{\scriptsize $\textrm{Hom}_{\Gamma_T}(F_T(t),\delta_*)$} (b);
\draw [->>] (a) to (c);
\draw[right hook->] (c) to (b);

\end{tikzpicture}
\end{center}
This actually gives us the following commutative diagram:
\begin{center}
\begin{tikzpicture}[line cap = round, line join = round]
\node (a) at (-3, 2) {$\homs{}(t,u)$};
\node (b) at (3, 2) {$\homs{}(t, \Sigma^d u^*).$};
\node (c) at (0,0) {$\textrm{Hom}_{\Gamma_T}(F_T(t),\textrm{Im}\delta_*)$};

\draw [->] (a) to node[above]{$\homs{}(t,\delta)$} (b);
\draw [->>] (a) to (c);
\draw[right hook->] (c) to (b);

\end{tikzpicture}
\end{center}
Combining this with Proposition \ref{someProp}(i), we obtain the required equality. The proof of part (ii) uses the same arguments.
\end{proof}
\section{A Counterexample}
For a triangulated category $\mathscr{C}$ with two cluster tilting subcategories $\mathscr{T}$ and $\mathscr{U}$, we always have sign coherence in the $c$-vector; that is, for a given $u \in \mathscr{U}$ and for all $t \in \mathscr{T}$, either $\cvecU{}(t) \geq 0$ or $\cvecU{}(t) \leq 0$. We demonstrate an example here where this sign coherence is not achieved for a higher case.\\

We will be working with the Oppermann-Thomas $(d+2)$-angulated categories of type $A_n$. We will label them as $\mathscr{C}(A_n^d)$.

The following description of $\mathscr{C}(A_n^d)$ is a restatement of Propositions 3.12 and 6.1 and Lemma 6.6(2) in \cite{OppermannThomas}. We take the canonical cyclic ordering of the set $V = \{1, \ldots, n + 2d+1\}$, which it can be helpful to think of as the vertices of an $(n + 2d + 1)$-gon labelled in a clockwise direction. This means that for three points in our ordering $x, y, z$ such that $x < y < z$, if we start at $x$ and move clockwise, we will encounter first $y$ then $z$. It is worth noting that if we have $x < y < z$, then we also have that $y < z \leq x$ and $z \leq x <  y$. For a point $x$ in our ordering, we denote by $x^-$ the vertex of our polygon that is one step anticlockwise of $x$.\\

\begin{proposition}\label{propBij}
The indecomposable objects of $\mathscr{C}(A_n^d)$ are in bijection with subsets of $V$ that have size $d+1$ and contain no neighbouring vertices. We identify each indecomposable $X$ with its subset of $V$, and will write $X = \{x_0, x_1, \ldots, x_d\}$.
\end{proposition}
We see immediately that by setting $d=1$ in proposition \ref{propBij}, we obtain the traditional cluster category of type $A_n$.

Using the identification described in proposition \ref{propBij}, we can easily describe the action of the translation functor, and also how the indecomposable objects interact with one another.
\begin{proposition}\label{propFunc}
The translation functor simply shifts an indecomposable by one place; that is, if $X = \{x_0, x_1, \ldots, x_d\}$, then $\Sigma^d(X) = \{x_0^-, x_1^-, \ldots, x_d^-\}$.
\end{proposition}

\begin{definition}\label{intertwine}
For two indecomposable objects $X$ and $Y$ of $\mathscr{C}(A_n^d)$, we say that $X$ and $Y$ \textit{intertwine} if there is a labelling of $X = \{x_0, x_1, \ldots, x_d\}$ and of $Y=\{y_0, y_1, \ldots, y_d\}$ such that 
\begin{align*}
x_0 < y_0 < x_1 < y_1 < x_2 < \ldots < x_d < y_d < x_0.
\end{align*}
\end{definition}
We see that definition \ref{intertwine} is symmetric; we take $Y = Y'$, where we choose the labelling as $y_i' = y_{i-1}$ for $1 \leq i \leq d$ and $y_0' = y_d$. This gives us that
\begin{align*}
y_0' < x_0 < y_1' < x_1 < y_2' <\ldots < y_d' < x_d < y_0'
\end{align*}
as required. \\

For two indecomposable objects $X$ and $Y$ of $\mathscr{C}(A_n^d)$, either $\textrm{Hom}(X, Y) = 0$ or $\textrm{Hom}(X, Y) = K$; this is the same as in the classic cluster category. In fact, we have the following: 
\begin{proposition}[{\cite[Proposition~6.1]{OppermannThomas}}]\label{homProp}
For two indecomposable objects $X$ and $Y$ of $\mathscr{C}(A_n^d)$, we have $\textrm{Hom}(X, Y) = K$ if and only if $X$ and $\Sigma^{-d}(Y)$ intertwine. This is equivalent to $X$ and $Y$ having labellings such that the following is true:
\begin{align*}
x_0 \leq y_0 \leq x_1^{--} < x_1 \leq y_1 \leq x_2^{--} < \ldots < x_d \leq y_d \leq x_0^{--}.
\end{align*}
\end{proposition}

We may also speak to whether or not there is a factorisation of a non-zero homomorphism in $\mathscr{C}(A_n^d)$.
\begin{proposition}[{\cite[Proposition~3.12]{OppermannThomas}}]\label{propFact}
For two indecomposable objects $X$ and $Y$ of $\mathscr{C}(A_n^d)$ satisfying the condition in Proposition \ref{homProp}, a non-zero morphism $X \to Y$ factors through a third irreducible $Z$ if and only if there exists a labelling for $Z=\{z_0, z_1, \ldots, z_d\}$ such that
\begin{align*}
x_0 \leq z_0 \leq y_0, x_1 \leq z_1 \leq y_1, \ldots, x_d \leq z_d \leq y_d.
\end{align*}
\end{proposition}

It is also true, again due to \cite{OppermannThomas}, that our categories $\mathscr{C}(A_n^d)$ permit Oppermann-Thomas cluster tilting objects. By \cite[Theorem~2.4]{OppermannThomas} and \cite[Theorem~6.4]{OppermannThomas}, the sum $T$ of ${n + d - 1 \choose d}$ mutually non-intertwining indecomposable objects is an Oppermann-Thomas cluster tilting object. Combining this with \cite[Lemma~6.6]{OppermannThomas} gives us the following proposition:
\begin{proposition}\label{sizeProp}
The sum $T$ of ${n + d - 1 \choose d}$ mutually non-intertwining indecomposable objects of $\mathscr{C}(A_n^d)$ is an Oppermann-Thomas cluster tilting object. Moreover, this describes all such objects. These objects are maximal with respect to the non-intertwining property.
\end{proposition}

We now state the counterexample:
Let $\mathscr{C} = \mathscr{C}(A_3^3)$. We let $T$ be the object given by summing all of the indecomposables containing the vertex $1$, and $U$ be the object obtained by summing all of the indecomposables containing the vertex $3$. In both cases the indecomposables are obviously non-intertwining and there are ${5 \choose 3}$ of them, so by Proposition \ref{sizeProp} $T$ and $U$ are Oppermann-Thomas cluster tilting objects. Let $\mathscr{T} = \textrm{add }T$ and $\mathscr{U} = \textrm{add }U$ be the Oppermann-Thomas cluster tilting subcategories associated with these objects. \\

We set $u = (3,5,8,10) \in \mathscr{U}$. We will examine the action of the $c$-vector $\cvecU{}$ on $\mathscr{T}$. \\
We take the indecomposable $t_1 =(1,4,6,9) \in \mathscr{T}$. By Theorem \ref{lemmap1}, we can calculate $\cvecU{}(t_1)$ by  taking the coefficient of $[u]$ in $\indU{}(\Sigma^3 t_1)$ and multiplying by $(-1)^3 = -1$. By Proposition \ref{propFunc} we see that $\Sigma^3 t_1 = (3,5,8,10)$, which is equal to $u$. Then $\indU{}(\Sigma^3 t_1) = [u]$, so we have that $\cvecU{}(t_1) = -1$. \\

We take the indecomposable $t_2 =(1,5,7,9) \in \mathscr{T}$. Again, we calculate $\Sigma^3 t_2 = (4,6,8,10)$. This is not in $\mathscr{U}$, so we need to find the $5$-angle made up of objects in $\mathscr{U}$ that covers this to give us the index. We aim to do this using \cite[Theorem~6.3]{OppermannThomas}. Firstly, we select an element of $\mathscr{U}$ which intertwines $\Sigma^3 t_2$; we take $(3,5,7,9)$. This gives us the $5$-angle
\begin{align*}
(3,5,7,9) \to (3,5,7,10) \to (3,5,8,10) \to (3,6,8,10) \to \Sigma^3 t_2 \to \Sigma^3 (3,5,7,9).
\end{align*}
As each object in the angle except for $\Sigma^3 t_2$ is in $\mathscr{U}$, we have that
\begin{align*}
\indU(\Sigma^3 t_2) = -[(3,5,7,9)] + [(3,5,7,10)] - [(3,5,8,10)] + [(3,6,8,10)].
\end{align*}
It follows that $\cvecU{}(t_2) = 1$. This demonstrates that there exist indecomposables in Oppermann-Thomas cluster tilting subcategories for $d>1$ that do not have sign coherence in their $c$-vector. It also means that by Lemma \ref{coherenceLemma}, $u$ is not mutable in $\mathscr{U}$.
\bibliography{mybib}{}
\bibliographystyle{mybib}
School of Mathematics and Statistics, Newcastle University, Newcastle upon Tyne, NE1 7RU, United Kingdom \\
\textit{Email address}: j.reid4@ncl.ac.uk
\end{document}